\documentclass[10pt]{amsart}
\usepackage{amsmath,amssymb,amsthm}
\usepackage{graphicx}
\usepackage{color}
\newtheorem{theorem}{Theorem}    

\newtheorem{corollary}{Corollary}
\theoremstyle{definition}

\newtheorem{remark}[theorem]{Remark}
\newtheorem*{remark*}{Remark}

\DeclareMathOperator{\ind}{\rm{ind}}

\title{On homogeneous quasipositive links}
\author[T.Ito]{Tetsuya Ito}
\address{Department of Mathematics, Kyoto University, Kyoto 606-8502, JAPAN}
\email{tetitoh@math.kyoto-u.ac.jp}
 
\begin{document}

\begin{abstract}
We discuss when homogeneous quasipositive links are positive. In particular, we show that a homogeneous diagram of a quasipositive link whose number of Seifert circles is equal to the braid index is a positive diagram. 
\end{abstract}

\maketitle

An $n$-braid is \emph{quasipositive} if it is a product of conjugates of the standard generators $\sigma_{1},\sigma_{2},\ldots,\sigma_{n-1}$. A link $K$ in $S^{3}$ is 
\begin{itemize}
\item[--] \emph{quasipositive} if $K$ can be represented by the closure of a quasipositive braid. 
\item[--] \emph{positive} if $K$ can be represented by a positive diagram, a diagram without negative crossing.
\item[--] \emph{homogeneous} if $K$ can be represented by a homogenous diagram (see \cite{Cr} for definition)
\end{itemize}
A homogenous diagram are common generalization of positive diagram and alternating diagram, sharing various nice properties of such diagrams.

In \cite[Question (4)]{Baader1} Baader asked whether an alternating quasipositive link is positive. He also posed a more general question, whether a homogenous quasipositive link is positive \cite[Problems (i)]{Baader2}.

Recently, in \cite{Ore} Orevkov gave an affirmative answer to the first Baader's question for a class of alternating links which he call \emph{DHL link} (Diao-Hertyi-Liu link); An alternating link is DHL link if it has an alternating diagram whose number of Seifert circles is equal to the braid index.
In \cite{DHL} Diao-Hertyi-Liu gave a simple and clear characterization of such an alternating diagram; for an alternating diagram $D$ of a link $K$, the number of Seifert circles of $D$ is equal to the braid index of $K$ if and only if is there is no pair of Seifert circles connected by a single crossing \cite{DHL}.

In this note we discuss Baader's second question, the positivity of homogenous quasipositive links.

For a link $K$ let $SL(K)$ be the maximal self-linking number and $b(K)$ be the braid index. For a link diagram $D$ of $K$, we denote by $O(D)$ the number of Seifert circles of $D$, and denote by $c_{+}(D), c_{-}(D)$ the number of positive and negative crossings of $D$.
We define the \emph{self-linking number} of diagram $D$ by $sl(D)=-O(D)+c_{+}(D)-c_{-}(D)$ (this is the self-linking number of closed braid obtained from $D$ by Vogel-Yamada method \cite{Vo,Ya}).

\begin{theorem}
\label{theorem:main}
Let $D$ be a homogenous diagram of a quasipositive link $K$.
If $SL(K)=sl(D)$, $D$ is a positive diagram.
\end{theorem}

This generalizes Orevkov's theorem for homogenous case;

\begin{corollary}
\label{cor:main}
Let $D$ be a homogenous diagram of a quasipositive link $K$. If $O(D)=b(K)$\footnote{In \cite{HIK} we call a homogeneous link having such a homogeneous diagram \emph{strongly homogeneous.}}, then $D$ is a positive diagram.
\end{corollary}
\begin{proof}[Proof of Corollary \ref{cor:main}]
Following Vogel-Yamada method \cite{Vo,Ya}, one can convert $D$ into a closed $O(D)$-braid diagram $\beta_D$ without changing the writhe $c_+(D)-c_{-}(D)$.
Since $\beta_D$ is a closed braid with minimum braid index, by generalized Jones' conjecuture proven in \cite{DP,LM}, the closed braid $\beta_D$ viewed as a transverse link attains the maximal self-linking number $SL(K)$ hence $SL(K)=sl(D)$.
\end{proof}

\begin{proof}[Proof of Theorem \ref{theorem:main}]
In the following we assume that $D$ is non-split, since to prove the theorem it is sufficient to treat the case $D$ is non-split.

Abe \cite{Abe} showed that for a homogeneous non-split diagram $D$ of $K$, the Rassumussen invariant $s(K)$ of $K$ is given by\footnote{We remark that the Rasmussen's invariant $s(K)$ is originally defined for knots \cite{Ra}. For link case we use Beliakova-Wherli's generalization of Rasmussen invariant for links \cite{BW}. Although \cite{Abe} proves the formula only for knot case, his proof can be applied to link cases, as he noted and mentioned in \cite{AT} (if the diagram is non-split).} 
\[ s(K) = -O(D)+c_+(D)-c_{-}(D) + 2O_{+}(D)-1 = sl(D) + 2O_{+}(D)-1 \]
where $O_+(D)$ is the number of connected components of $D$ obtained by smoothing all the negative crossings.

Since $K$ is quasipositive, slice Bennequin inequality gives an equality hence
\[ SL(K)= s(K)-1 = -\chi_4(K).\]
Here $\chi_4(K)$ the maximum euler characteristic of smooth oriented surface in $B^{4}$ whose boundary is $K$ (see \cite{Shu}). 
Therefore by our hypothesis $sl(K)=sl(D)$, we get $O_+(D)=1$. For a homogeneous non-split diagram $D$, $O_+(D)=1$ happens only if $c_{-}(D)=0$. Therefore $D$ is positive.

\end{proof}
 
Actually our argument shows that a homogeneous diagram of a quasipositive link should be close to a positive diagram in the following sense.

\begin{theorem}
Let $D$ be a non-split homogeneous diagram of a quasipositive link $K$. 
Then $O(D)-b(K) \geq O_{+}(D)-1$ holds.
\end{theorem}
\begin{proof}
By generalized Jones' conjecture proven in \cite{DP,LM}
\[ sl(D)\geq SL(K)-2(O(D)-b(K)).\]
Hence
\begin{align*}
sl(D) & \geq SL(K)-2(O(D)-b(K)) =  sl(D) +2(O_+(D)-1)-2(O(D)-b(K))
\end{align*}
\end{proof}

Unfortunately, in general it is not true that a homogeneous (alternating) link $K$ has a homogeneous (alternating) diagram $D$ with $SL(K)=sl(D)$ (see Remark \ref{rem} below).

To overcome the defect, we give a different class of homogeneous diagrams
having an affirmative answer to Baader's question.
In \cite{MP} Murasugi-Prytycki gave an refinement of famous Motron-Franks-Williams inequalities.
\begin{equation}
\label{eqn:MP-MFW} \min \deg_{v} P_{K}(v,z) \geq -O(D)+c_{+}(D)-c_{-}(D)+1 + 2 \ind_-(D)\end{equation}
\begin{equation}
\label{eqn:MP-MFW2} \max \deg_{v} P_{K}(v,z) \leq O(D)+c_{+}(D)-c_{-}(D)-1 - 2 \ind_+(D)\end{equation}
\begin{equation}
\label{eqn:MP-MFW3} b(K) \leq O(D)-\ind(D)
\end{equation}
We call the last inequaility (\ref{eqn:MP-MFW3}) of the braid index the \emph{MP inequality} (Murasugi-Prytycki inequality).

Here $P_{K}(v,z)$ is the HOMFLY polynomial of $K$, and $\ind_{\pm}(D),\ind(D)$ are the index of the Seifert graph of $D$ (see \cite{MP} for definitions).
In particular, they proved for a given knot diagram $D$ of $K$ one can find a diagram $D'$ of the same knot $K$ so that
\[ sl(D') = sl(D) + 2\ind_{-}(D),\quad O(D') = O(D)-\ind_-(D)\]
\[ sl(D') = sl(D), \quad O(D') = O(D)-\textrm{ind}_{+}(D).\]
For general link diagram $D$,
\[ \textrm{ind}_{+}(D)+\textrm{ind}_{-}(D) \geq \textrm{ind}(D).\]
When $D$ is homogeneous, $\ind(D)=\ind_-(D)+\ind_+(D)$, and there is a diagram $D'$ such that 
\[ sl(D') = sl(D) + 2\ind_{-}(D),\quad O(D') = O(D)-\ind(D)\]

\begin{remark}
\label{rem}
If $D$ is a reduced alternating diagram of $\ind_{-}(D)>1$, 
by (\ref{eqn:MP-MFW}) $sl(D)=SL(K)$ cannot happen. Since for a reduced alternating link $K$, the writhe of a reduced alternating diagram is invariant of an alternating link, such a link does not have an alternating diagram $D$ with $sl(D)=SL(K)$.
\end{remark}

Using these results we give another class of links having an affirmative answer to Baader's question.

\begin{theorem}
\label{theorem:main2}
Let $D$ be an irreducible homogeneous diagram of a quasipositive link. If $sl(D)+2\ind_{-}(D) = SL(K)$, then $D$ is positive.
\end{theorem}
\begin{proof}
As in Theorem \ref{theorem:main}, it is sufficient to treat the case $D$ is non-split. Then by quasipositivity $SL(K)=sl(D)+2(O_+(D)-1)$, we have $\ind_{-}(D) = O_{+}(D)-1$. However, since $D$ is irreducible this is possible only if $\ind_{-}(D)=0$. This follows from the following definitions, results, and observations:
\begin{itemize}
\item the Seifert graph $\Gamma(D)$ of homogeneous diagram $D$ is written as the $\ast$-product $\Gamma(D_1)\ast \Gamma(D_2) \ast \cdots \ast \Gamma(D_{n})$, where each diagram $D_i$ is a positive or negative special diagram. Moreover, if $D$ is reduced, no $D_{i}$ contains a vertex of valence one (definition of homogeneous diagram \cite{Cr}).
\item $\ind_{\pm}(D) = \sum_{i=1}^{n} \ind_{\pm}(D_i)$, since $\Gamma(D_i)$ is bipartite \cite[Theorem 2.4]{MP}.
\item If $D$ is connected, $O_{+}(D)-1 = \sum_{i=1}^{n}(O_{+}(D_i)-1)$ (follows from the definition of $\ast$-product).
\item For a negative special diagram $D_i$, $\ind_{-}(D_i) > O_{+}(D_i) - 1$ whenever its Seifert graph $\Gamma(D_i)$ has no vertices of valence one (follows from the definition of index).
\end{itemize}
Thus $O_+(D)=1$ and $c_-(D)=0$, as desired.
\end{proof}

In particular, we have the following.

\begin{corollary}
Let $D$ be an irreducible homogeneous diagram of a quasipositive link. If the MP inequality of the braid index (\ref{eqn:MP-MFW3}) is equality, then $D$ is positive.
\end{corollary}

In \cite{MP} it is conjectured that for alternating diagram $D$, the MP inequality of the braid index is an equality. Our discussion so far shows that an affirmative answer to Murasugi-Prytycki's conjecture implies an affirmative answer to Baader's first question.

Finally, although it is irrelevant to our study of homogeneous quasipositive links, we give an improvement of the MP inequality (\ref{eqn:MP-MFW3}) which is interesting in its own right.

\begin{theorem}[Refined MP inequality]
\label{theorem:refined-MP}
For any link diagram $D$ of a link $K$,
\[ b(K) \leq O(D)-(\ind_{+}(D)+\ind_{-}(D)) \]
\end{theorem}
\begin{proof}
Let $\overline{D}$ and $\overline{K}$ be the mirror image of $D$ and $K$, respectively. Since $\ind_{-}(\overline{D})=\ind_{+}(D)$ by definition, 
\[ sl(D)+2\ind_{-}(D) \leq SL(K), \quad sl(\overline{D})+2\ind_{+}(D) \leq SL(\overline{K}) \]
hence
\[ (sl(D)+sl(\overline{D}))+2 ( \ind_{+}(D)+\ind_{-}(D) ) \leq SL(K)+SL(\overline{K}) \]
On the other hand, (generalized) Jones' conjecture implies $SL(K)+SL(\overline{K})=-2b(K)$. Since $sl(D)+sl(\overline{D}) = -2O(D)$ by definition, we conclude
\[ -2O(D) + 2 ( \ind_{+}(D)+\ind_{-}(D) ) \leq -2b(K).\]
\end{proof}

Unlike the proof of original MP inequality where they explicitly construct a link diagram $D'$ with $O(D') = O(D)-\ind(D)$ from $D$, our proof of Theorem \ref{theorem:refined-MP} does not give a link diagram $D'$ with $O(D') = O(D)-(\ind_{+}(D)+\ind_{-}(D))$. Finding how to get such a link diagram $D'$ from $D$ would be an interesting question.

\begin{corollary}
For a diagram $D$ of a link $K$, 
If equalities hold in (\ref{eqn:MP-MFW}) and (\ref{eqn:MP-MFW2}), then $b(K)=O(D)-\ind_{-}(D)-\ind_{+}(D)=\frac{1}{2}\textrm{span}_{v}P_K(v,z) +1$ (in particular, the Morton-Franks-Williams inequality is equality).
\end{corollary}
\begin{proof}
If equalities hold in (\ref{eqn:MP-MFW}) and (\ref{eqn:MP-MFW2}) then
$\frac{1}{2}\textrm{span}_{v}P_K(v,z) +1 = O(D)-(\ind_{+}(D)+\ind_{-}(D))$.
By the Morton-Franks-Williams inequality and refined MP inequality 
\[ \frac{1}{2}\textrm{span}_{v}P_K(v,z) +1 \leq b(K) \leq  O(D)-(\ind_{+}(D)+\ind_{-}(D))\]
\end{proof}

\section*{Acknowledgement}
The author has been partially supported by JSPS KAKENHI Grant Number 19K03490,16H02145. He is grateful to Stepan Yu Orevkov for inspiring the study of positivity of homogeneous quasipositive links.

\end{document}